\pdfoutput=1
\documentclass{article}
\usepackage{amsmath,amssymb,amsthm}

\newenvironment{acknowledgment}[1]
  {\par\medskip\footnotesize\noindent\textbf{\MakeUppercase{#1}}\ \ignorespaces}{\par}

\theoremstyle{plain}
\newtheorem{theorem}{Theorem}

\newtheorem{claim}[theorem]{Claim}

\theoremstyle{definition}

\newtheorem{remark}[theorem]{Remark}
\usepackage{lmodern}

\title{A Sharp Norm Inequality and Buzano's Inequality via Determinants}
\author{Jose Antonio Lara Benitez\thanks{Rice University, Houston, TX. Present address: Lawrence Berkeley National Laboratory; Department of Statistics, University of California, Berkeley; and International Computer Science Institute, Berkeley, CA. Email: \texttt{jlarabenitez@lbl.gov}; \texttt{jlarabenitez@berkeley.edu}; \texttt{jlarabenitez@icsi.berkeley.edu}.}}
\date{}

\begin{document}
\maketitle

\begin{abstract}
We give a short linear-algebraic proof of the inequality
\[
\|x\|_1\,\|x\|_\infty \le \frac{1+\sqrt{n}}{2}\,\|x\|_2^2,
\]
valid for every \(x\in\mathbb{R}^n\). This inequality relates three fundamental norms on finite-dimensional spaces and has applications in optimization and numerical analysis. Our proof exploits the determinantal structure of a parametrized family of quadratic forms, and we show the constant $(1+\sqrt{n})/2$ is optimal. The inequality is a special case of Buzano's inequality, and the same method also proves Buzano's inequality for real vectors.
\end{abstract}

\vspace{10pt}

\vspace{6pt}

\section{Introduction.}

Inequalities relating different norms on $\mathbb{R}^n$ are fundamental in analysis, optimization, and numerical linear algebra. Among the most basic are the equivalences between the $\ell^1$, $\ell^2$, and $\ell^\infty$ norms. While the triangle and Cauchy--Schwarz inequalities immediately yield $\|x\|_2 \le \|x\|_1 \le \sqrt{n}\,\|x\|_2$ and $\|x\|_\infty \le \|x\|_2 \le \sqrt{n}\,\|x\|_\infty$, mixed products of these norms are less transparent.

The inequality
\begin{equation}\label{eq:main}
\|x\|_1\,\|x\|_\infty \le \frac{1+\sqrt{n}}{2}\,\|x\|_2^2
\end{equation}
provides a sharp upper bound for the product $\|x\|_1\,\|x\|_\infty$ in terms of the squared Euclidean norm. Bounds of this kind arise in the analysis of coordinate descent methods \cite[Section~1, (1.3)--(1.4)]{nesterov2012efficiency} and of $\ell^1$-regularized problems in statistics and machine learning \cite[Chapter~10, pp.~269--329]{beck2017first}. Inequality \eqref{eq:main} itself, applied to the eigenvalues of a positive semidefinite matrix, is Lemma~6.1 of the recent preprint \cite{zhang2026moduli}, where it is proved by calculus and yields a sharp triangle inequality for the Frobenius norm.

In this note we give an elementary proof using only linear algebra: we reformulate the problem as asking when a certain parametrized quadratic form is nonnegative, construct its representing matrix explicitly, and compute its determinant via elementary column operations. The optimal constant emerges naturally as the root of a quadratic equation. Inequality \eqref{eq:main} is a special case of Buzano's inequality \cite{buzano1974,fujii1993}; in Section~\ref{sec:buzano} we show how the same determinantal method proves that inequality over $\mathbb{R}^n$. For comparison, Remark~\ref{rem:lagrange} sketches a proof of \eqref{eq:main} by Lagrange multipliers.

Our approach illustrates a general technique: many norm inequalities can be recast as questions about positive semidefiniteness, which can then be settled by determinantal calculations \cite[Section~7.2, pp.~438--447]{horn2012matrix}. We hope this method proves useful in other contexts.

\section{Proof of the Inequality.}

Let \(x\in\mathbb{R}^n\). Recall the standard definitions
\[
\|x\|_1=\sum_{i=1}^n|x_i|,\qquad
\|x\|_2^2=\sum_{i=1}^n x_i^2,\qquad
\|x\|_\infty=\max_{1\le i\le n}|x_i|.
\]
Without loss of generality assume \(|x_1|=\|x\|_\infty\) (relabel coordinates if necessary) and set \(u_i=|x_i|\) for \(i=1,\dots,n\). Then proving \eqref{eq:main} reduces to showing
\begin{equation}\label{eq:uform}
u_1\sum_{i=1}^n u_i \le \frac{1+\sqrt{n}}{2}\sum_{i=1}^n u_i^2
\end{equation}
for all $u\in\mathbb{R}^n$ with $u_1,\dots,u_n\ge0$ and $u_1 = \max_i u_i$.

Let \(C>0\) be a constant to be determined. If $C=(1+\sqrt{n})/2$, then the inequality \eqref{eq:uform} is equivalent to the nonnegativity of the quadratic form
\begin{equation}\label{eq:QC}
Q_C(u)=C\sum_{i=1}^n u_i^2 - u_1\sum_{i=1}^n u_i
\end{equation}
for all $u$ in the first orthant with $u_1\ge u_i$ for all $i$.

Expanding \eqref{eq:QC}, we obtain
\begin{equation}\label{eq:expanded}
Q_C(u) = (C-1)u_1^2 - u_1u_2 - \cdots - u_1u_n + Cu_2^2 + \cdots + Cu_n^2.
\end{equation}
From this expression, the symmetric matrix associated with the quadratic form $Q_C$ is
\[
Q_n(C)=
\begin{pmatrix}
C-1 & -\tfrac12 & -\tfrac12 & \cdots & -\tfrac12\\[1mm]
-\tfrac12 & C & 0 & \cdots & 0\\
-\tfrac12 & 0 & C & \ddots & \vdots\\
\vdots & \vdots & \ddots & \ddots & 0\\
-\tfrac12 & 0 & \cdots & 0 & C
\end{pmatrix},
\]
so that $Q_C(u) = u^T Q_n(C) u$. If $Q_n(C)$ is positive semidefinite, then $Q_C(u)\ge 0$ for all $u$, establishing \eqref{eq:uform} if $C=(1+\sqrt{n})/2$. We will determine the smallest $C$ for which $Q_n(C)$ is positive semidefinite.

Let \(D_k(C)\) denote the determinant of the leading \(k\times k\) principal submatrix of \(Q_n(C)\). This submatrix has the same structure as $Q_k(C)$.

\begin{claim}\label{claim:determinant}
For $1\le k\le n$,
\[
D_k(C) = C^{k-2} \left( C^2 - C - \frac{k-1}{4} \right).
\]
\end{claim}

\begin{proof}
We compute the determinant of $Q_k(C)$ by performing elementary column operations to triangularize the matrix. Let $\mathbf{v}_j$ denote the $j$-th column of $Q_k(C)$. We replace the first column $\mathbf{v}_1$ with
\[
\mathbf{v}_1' = \mathbf{v}_1 + \sum_{j=2}^k \frac{1}{2C} \mathbf{v}_j.
\]
For any row $i \ge 2$, the entry in the first column is $-\frac{1}{2}$, and the entry in the $j$-th column is $C$ if $i=j$ and $0$ otherwise. Thus, the new entry at position $(i,1)$ is
\[
-\frac{1}{2} + \frac{1}{2C}(C) = 0.
\]
This clears all entries in the first column below the diagonal. The entry at position $(1,1)$ becomes
\[
(C-1) + \sum_{j=2}^k \frac{1}{2C}\left(-\frac{1}{2}\right) = (C-1) - \frac{k-1}{4C}.
\]
The resulting matrix is upper triangular. The determinant is the product of its diagonal entries:
\[
D_k(C) = \left( (C-1) - \frac{k-1}{4C} \right) \cdot C^{k-1} = C^{k-2} \left( C(C-1) - \frac{k-1}{4} \right),
\]
which simplifies to the claimed formula.
\end{proof}

For $k=n$, we have
\[
D_n(C) = C^{n-2}\left(C^2-C-\tfrac{n-1}{4}\right).
\]
The quadratic $C^2-C-\tfrac{n-1}{4}$ has roots
\[
C = \frac{1\pm\sqrt{1+n-1}}{2} = \frac{1\pm\sqrt{n}}{2}.
\]
Set
\[
\varphi_n=\frac{1+\sqrt{n}}{2}>0.
\]
Then
\[
D_n(C) = C^{n-2}\left(C-\varphi_n\right)\left(C-\varphi_n+\sqrt{n}\right).
\]

For any $k\le n$, substituting $\varphi_n$ into the closed form:
\[
D_k(\varphi_n)=\varphi_n^{k-2}\left(\varphi_n^2-\varphi_n-\tfrac{k-1}{4}\right) = \varphi_n^{k-2}\cdot\frac{n-k}{4}\ge 0,
\]
with equality when $k=n$ and strict inequality for $k<n$.

Furthermore, note that $\varphi_k := \tfrac{1+\sqrt{k}}{2}$ is increasing in $k$, so $\varphi_n = \max_{k\le n} \varphi_k$. For $C>\varphi_n$, all leading principal minors $D_k(C)>0$ for $k=1,\ldots,n$. Sylvester's criterion states that a symmetric square matrix is positive definite if all of its leading principal minors are positive (see \cite[Theorem~7.2.5]{horn2012matrix}). Thus $Q_n(C)$ is positive definite for all $C>\varphi_n$.

\begin{claim}\label{claim:boundary}
$Q_n(\varphi_n)$ is positive semidefinite.
\end{claim}

\begin{proof}
We prove that all eigenvalues of $Q_n(\varphi_n)$ are nonnegative. Since $Q_n(C)$ is symmetric, its eigenvalues are real and depend continuously on $C$ (see \cite[Section~4.1, pp.~227--233, and Appendix~D, pp.~563--564]{horn2012matrix}).

As shown above, for all $C>\varphi_n$, the matrix $Q_n(C)$ is positive definite, so all its eigenvalues are strictly positive. Thus all the eigenvalues of the symmetric matrix $Q_n(\varphi_n)$ are nonnegative, which implies that $Q_n(\varphi_n)$ is positive semidefinite.
\end{proof}

By Claim~\ref{claim:boundary}, $Q_{\varphi_n}(u)\ge 0$ for all $u\in\mathbb{R}^n$. Therefore
\[
u_1\sum_{i=1}^n u_i \le \varphi_n\sum_{i=1}^n u_i^2.
\]
Recalling $u_i=|x_i|$, $u_1=\|x\|_\infty$, and $\varphi_n=(1+\sqrt{n})/2$, this establishes \eqref{eq:main}. \qed

\section{Optimality of the Constant.}

For $n=1$, inequality \eqref{eq:main} is an equality, so assume $n\ge2$. To show the constant $\varphi_n = (1+\sqrt{n})/2$ is optimal, it suffices to find a non-zero vector $u$ such that equality holds in \eqref{eq:uform}. This corresponds to finding a non-trivial solution to $Q_n(\varphi_n)u = 0$.

Since $D_n(\varphi_n)=0$, the matrix $Q_n(\varphi_n)$ is singular and has a non-trivial kernel. Due to the symmetry of the matrix indices $2,\dots,n$, we look for a vector of the form $u = (1, t, t, \dots, t)^T$.

Considering the second row of the system $Q_n(\varphi_n)u = 0$, we have
\[
-\frac{1}{2}(1) + \varphi_n t = 0 \implies t = \frac{1}{2\varphi_n}.
\]
Rows $3,\dots,n$ give the same equation, and the first row, $(\varphi_n-1)-\tfrac{n-1}{2}\,t=\tfrac{1}{\varphi_n}\bigl(\varphi_n^2-\varphi_n-\tfrac{n-1}{4}\bigr)=0$, is then satisfied as well.
Substituting $\varphi_n = (1+\sqrt{n})/2$, we obtain
\[
t = \frac{1}{1+\sqrt{n}} = \frac{\sqrt{n}-1}{n-1}.
\]
Note that for $n \ge 2$, we have $\varphi_n > 1$, which implies $t < 1/2 < 1$. Thus $u_1 = 1$ is indeed the maximum entry $\|u\|_\infty$, consistent with our assumptions.

This vector $u$ yields $Q_{\varphi_n}(u)=0$, proving that equality is attained and the constant cannot be improved. Indeed, one can see from the expression \eqref{eq:expanded} that $Q_C(u)=Q_{\varphi_n}(u)+(C-\varphi_n)\sum_{i=1}^n u_i^2<0$ for any $C<\varphi_n$.

\begin{remark}[Comparison with Lagrange multipliers]\label{rem:lagrange}
Maximize $F(u)=u_1(u_1+u_2+\cdots+u_n)$ on the unit sphere of $\mathbb{R}^n$. At a maximizer, the Lagrange multiplier rule gives $\nabla F(u)=2\mu u$, or
\[
2u_1+u_2+\cdots+u_n=2\mu u_1,\qquad
u_1=2\mu u_j\quad (j\ge2).
\]
Multiplying by the corresponding coordinates and adding gives $F(u)=\mu$. Since $F(e_1)=1$, we have $\mu\ge1$. Thus $u_j=u_1/(2\mu)$ for $j\ge2$, and $u_1\ne0$, since otherwise $u=0$. Substitution into the first equation yields $\mu^2-\mu-(n-1)/4=0$. The maximum is therefore $\mu=(1+\sqrt{n})/2$, and homogeneity gives \eqref{eq:main}. This is the same quadratic equation obtained from $D_n(C)$ in Section~2.
\end{remark}

\section{Buzano's Inequality.}\label{sec:buzano}

Section~2 treats the product of two linear forms, $u_1$ and $\sum_i u_i$. Applying the same determinantal method to an arbitrary pair gives Buzano's inequality \cite{buzano1974,fujii1993}: for all $a,b,z\in\mathbb{R}^n$,
\begin{equation}\label{eq:buzano}
|\langle a,z\rangle\langle z,b\rangle|\le\tfrac12\bigl(|\langle a,b\rangle|+\|a\|_2\|b\|_2\bigr)\|z\|_2^2 .
\end{equation}
Taking $a=e_1$, $b=(1,\dots,1)$, and $z=(|x_1|,\dots,|x_n|)$, with $|x_1|=\|x\|_\infty$, recovers \eqref{eq:main}.

If $a=0$ or $b=0$, both sides of \eqref{eq:buzano} vanish; for $n=1$, they are equal. We may therefore assume $a,b\ne0$ and $n\ge2$. Choose orthonormal coordinates in which $a=\alpha e_1$ and $b=\beta e_1+\gamma e_2$, with $\alpha>0$ and $\gamma\ge0$. This change preserves inner products and norms. The quadratic form $C\|z\|_2^2-\langle a,z\rangle\langle z,b\rangle$ then has matrix $\operatorname{diag}(B(C),CI_{n-2})$, where
\[
B(C)=\begin{pmatrix} C-\alpha\beta & -\tfrac12\alpha\gamma\\[1mm] -\tfrac12\alpha\gamma & C\end{pmatrix}.
\]
Its determinant is $C^2-\alpha\beta\,C-\tfrac14\alpha^2\gamma^2$, whose larger root is $C_0=\tfrac12\bigl(\alpha\beta+\alpha\sqrt{\beta^2+\gamma^2}\bigr)=\tfrac12\bigl(\langle a,b\rangle+\|a\|_2\|b\|_2\bigr)$. The leading principal minors are $C-\alpha\beta$ and $C^{k-2}\det B(C)$ for $2\le k\le n$. Since $C_0\ge\max\{\alpha\beta,0\}$, all are positive for $C>C_0$. Sylvester's criterion gives $\langle a,z\rangle\langle z,b\rangle\le C\|z\|_2^2$; letting $C\downarrow C_0$ gives the same bound with $C_0$.

Applying this bound with $-b$ in place of $b$ bounds the negative of the product. The larger of the two constants is $\tfrac12\bigl(|\langle a,b\rangle|+\|a\|_2\|b\|_2\bigr)$, which proves \eqref{eq:buzano}.

Equality in \eqref{eq:buzano} is attained by $z=\|b\|_2a+\|a\|_2b$ when $\langle a,b\rangle\ge0$, and by $z=\|b\|_2a-\|a\|_2b$ when $\langle a,b\rangle<0$, as direct substitution shows. For $a=e_1$ and $b=(1,\dots,1)$, the first vector is proportional to the extremizer in Section~3.

\begin{acknowledgment}{Acknowledgment.}
The author thanks Anastasios Kyrillidis for bringing this inequality to the author's attention in one of his courses, and the editor and the two anonymous referees for their helpful suggestions.
\par\noindent{\bfseries\MakeUppercase{Author contribution and conflict of interest.}}\ J.\,A.\ Lara Benitez developed the proof and wrote the manuscript, and declares no conflict of interest.
\end{acknowledgment}
\begin{acknowledgment}{Declaration of generative AI use.}
The mathematics in this note, including the determinantal proof and the extremizer, was developed by the author without the use of AI tools. During the revision, generative AI tools (Claude, Anthropic) were used to check the arguments and the numerical computations, to search the literature, and to edit the text. The author reviewed all of this output and takes full responsibility for the content.
\end{acknowledgment}

\bibliographystyle{vancouver}
\bibliography{ref}

\end{document}